\documentclass[a4paper,12pt]{amsproc}
\usepackage{amsmath,amsfonts,amssymb,amsthm,anysize,dsfont,enumerate}

\newcommand{\tr}{\text{Tr}}
\newcommand\Z{{\mathbb Z}}
\newcommand\F{{\mathbb F}}
\newcommand\Tr{{\mathrm{Tr}}}
\newcommand\Norm{{\mathrm{Norm}}}
\newcommand\Cay{{\mathrm{Cay}}}
\newcommand\PG{\mathsf{PG}}
\newcommand\cQ{{\mathcal Q}}
\newcommand\cM{{\mathcal M}}

\newcommand\la{\langle}
\newcommand\ra{\rangle}
\newcommand\sgn{\textup{sgn}}
\newcommand{\KD}{\mathds{1}_{P^\perp}}   
\newcommand{\dualG}{{\widehat G}}
\newcommand{\herm}{\mathsf{H}}

\theoremstyle{plain}
\newtheorem{theorem}{Theorem}[section]
\newtheorem{problem}[theorem]{Problem}
\newtheorem{lemma}[theorem]{Lemma}
\newtheorem{corollary}[theorem]{Corollary}
\newtheorem{construction}[theorem]{Construction}
\newtheorem{result}[theorem]{Result}
\numberwithin{equation}{section}

\theoremstyle{remark}
\newtheorem{remark}[theorem]{Remark}

\usepackage{url, hyperref}
\hypersetup{citecolor=blue, linkcolor=blue, colorlinks=true}

\makeatletter
\let\@@pmod\pmod
\DeclareRobustCommand{\pmod}{\@ifstar\@pmods\@@pmod}
\def\@pmods#1{\mkern4mu({\operator@font mod}\mkern 6mu#1)}
\makeatother

\renewcommand\le{\leqslant}
\renewcommand\ge{\geqslant}

\begin{document}

\title{A New Infinite Family of Hemisystems of the Hermitian Surface}
\author{John Bamberg}
\address{ %
Centre for the Mathematics of Symmetry and Computation\\
School of Mathematics and Statistics\\
The University of Western Australia\\
35 Stirling Highway, Crawley, W.A. 6009, Australia}
\email{John.Bamberg@uwa.edu.au}
\thanks{The first author acknowledges the support of the 
Australian Research Council Future Fellowship FT120100036.}

\author{Melissa Lee}
\address{ %
Centre for the Mathematics of Symmetry and Computation\\
School of Mathematics and Statistics\\
The University of Western Australia\\
35 Stirling Highway, Crawley, W.A. 6009, Australia}
%\email{20945795@student.uwa.edu.au}
\email{melissa.lee@research.uwa.edu.au}
\thanks{The second author acknowledges the support of a 
Hackett Postgraduate Research Scholarship.}

\author{Koji Momihara}
\address{ %
Department of Mathematics\\
Faculty of Education\\
Kumamoto University\\
2-40-1 Kurokami, Kumamoto 860-8555, Japan}
\email{momihara@educ.kumamoto-u.ac.jp}
\thanks{The third author acknowledges the support by 
JSPS under Grant-in-Aid for Young Scientists (B) 25800093 and Scientific Research (B) 15H03636.}

\author{Qing Xiang}
\address{ %
Department of Mathematical Sciences\\
University of Delaware\\
Newark DE 19716, USA
}
\email{xiang@math.udel.edu}

\subjclass[2010]{05B25 (primary), 05E30, 51E12 (secondary)}
\keywords{Hemisystem, partial quadrangle, strongly regular graph}

\begin{abstract}
In this paper, we construct an infinite family of hemisystems of the Hermitian surface $\herm(3,q^2)$.
In particular, we show that for every odd prime power $q$ congruent to $3$ modulo $4$, there
exists a hemisystem of $\herm(3,q^2)$ admitting $C_{(q^3+1)/4} : C_3$.
\end{abstract}

% 2000 MSC numbers
% 05B25 Finite geometries
% 05E30 Association schemes, strongly regular graphs
% 51E12 Generalized quadrangles, generalized polygons

\maketitle

\section{Introduction} A hemisystem of a generalized quadrangle of order $(q^2,q)$, $q$ odd, is a
set of lines $\mathcal{H}$ containing half of the lines on every point. Hemisystems are of interest
because they give rise to strongly regular graphs, partial quadrangles and 4-class imprimitive
cometric $Q$-antipodal association schemes that are not metric \cite{Dam:2013aa}. For a prime power
$q$, the \emph{classical} generalized quadrangle of order $(q^2,q)$ is the Hermitian surface
$\mathsf{H}(3,q^2)$ with automorphism group $\mathsf{P\Gamma U}(4,q)$. An \emph{$m$-cover} of a
generalized quadrangle is a set of lines such that every point is incident with $m$ lines from this
set. For instance, a \emph{spread} is a $1$-cover and a hemisystem of $\herm(3,q^2)$ has
$m=(q+1)/2$. It was shown by Segre \cite{Segre:1965aa} that the only nontrivial $m$-covers of
$\herm(3,q^2)$, $q$ odd, are hemisystems and he gave an example on $\herm(3,3^2)$. Bruen and
Hirschfeld \cite{Bruen:1978qf} showed that there are no nontrivial $m$-covers of $\herm(3,q^2)$ when
$q$ is even. Segre's example of a hemisystem remained the only example of a hemisystem for thirty
years, and so it was reasonable for Thas \cite{Thas:1995aa} in 1995 to pose a conjecture that
Segre's example was the only one. In 2005, Cossidente and Penttila \cite{Cossidente:2005aa}
disproved this conjecture by showing the existence of a hemisystem on $\herm(3,q^2)$, $q$ odd,
admitting $\mathsf{P\Omega}^-(4,q)$ for each odd prime power $q$. Cossidente and Penttila
\cite[Remark 4.4]{Cossidente:2005aa} also found by computer search a hemisystem of $\herm(3,7^2)$
with full stabilizer in $\mathsf{P\Gamma U}(4,7)$ a metacyclic group of order $516$, and a
hemisystem of $\herm(3, 9^2)$, with full stabilizer in $\mathsf{P\Gamma U}(4,9)$ a metacyclic group
of order $876$. Bamberg, Giudici and Royle \cite[Section 4.1]{Bamberg:2013aa} found that the pattern
continues (except, curiously, for $q$ congruent to $1$ modulo $12$) by finding for each $q\in
\{11,17,19, 23, 27\}$ a hemisystem of $\herm(3,q^2)$ admitting a cyclic group of order $q^2-q+1$. In
this paper, we construct an infinite family of hemisystems that generalize the examples above where
$q$ is congruent to $3$ modulo $4$.

\begin{theorem}\label{main}
 There is a hemisystem of $\herm(3,q^2)$ for every prime power $q \equiv  3\pmod{4}$, each admitting $C_{(q^3+1)/4} : C_3$.
\end{theorem}
 
We also prove that this infinite family of hemisystems yields new hemisystems beyond the small known
examples. We note from \cite{Bamberg:2013aa} that for $q=3$, our construction gives a hemisystem
that is projectively equivalent to Segre's hemisystem, and for $q=7$, the hemisystem is equivalent
to that given in \cite[Remark 4.4]{Cossidente:2005aa}.

Our construction is based on initially identifying a hemisystem of $\herm(3,q^2)$ with its
dual\footnote{By interchanging the roles of points and lines of a generalized quadrangle we obtain
another generalized quadrangle, the \emph{dual}.} set of points of the elliptic quadric
$\cQ^-(5,q)$. A hemisystem (of points) of $\cQ^-(5,q)$ in this context is defined as a set of points
$\mathcal{M}$ containing half of the points on every line. The technique used to construct these
hemisystems is remarkably similar to the technique used in \cite{Feng:2015aa} to construct
$\frac{(q^2-1)}{2}$-tight sets of $\cQ^+(5,q)$, otherwise known via the Klein correspondence as
\emph{Cameron-Liebler line classes} of $\PG(3,q)$ with parameter $\frac{(q^2-1)}{2}$. The
construction in this paper is essentially a cyclotomic construction, that is, the hemisystem of
$\cQ^-(5,q)$ we are going to construct is a union of cyclotomic classes of
$\mathbb{F}_{q^6}^*=\mathbb{F}_{q^6}\setminus\{0\}$. The first step is to give a finite field model
of $\cQ^-(5,q)$ for $q\equiv 3\pmod{4}$: we view $\mathbb{F}_{q^6}$ as a $6$-dimensional vector
space over $\mathbb{F}_q$ and define $\cQ^-(5,q)$ (using the underlying vectors instead of
projective points) as the nonzero vectors in the zero-set of the quadratic form
$\tr_{q^3/q}(x^{q^3+1})$ defined on $\mathbb{F}_{q^6}$, where $\tr_{q^3/q}$ is the trace from
$\F_{q^3}$ to $\F_q$. Note that in this setting, $\cQ^-(5,q)$ is also a union of $4(q+1)$ cyclotomic
classes of index $4(q^2+q+1)$ of $\mathbb{F}_{q^6}^*$. Of course, this field model and the
cyclotomic interpretation of $\cQ^-(5,q)$ are well known. In order to construct a hemisystem of
$\cQ^-(5,q)$, we need to choose half of the cyclotomic classes involved in the definition of
$\cQ^-(5,q)$. The difficulty lies in deciding which half to choose. We overcome this difficulty by
using a partition of a conic in $\PG(2,q)$ first discovered in \cite{Feng:2015aa}. This partition of
a conic in $\PG(2,q)$ gives us a way to choose half of the cyclotomic classes involved in the
definition of $\cQ^-(5,q)$, yielding a hemisystem of points of $\cQ^-(5,q)$. The proof that our
choice indeed works for each $q \equiv 3\pmod{4}$ relies on computations of (additive) character
values of the subset of chosen vectors using Gauss sums.

The paper is structured as follows: in Section 2, we give the requisite background on generalized
quadrangles (particularly the elliptic quadric $\cQ^-(5,q)$), $m$-ovoids, strongly regular graphs,
Cayley graphs, and Gauss sums. Then in Sections \ref{section:conic} and \ref{section:construction}
we work towards describing the new hemisystems and giving a proof of Theorem \ref{main}. Finally, we
show that the hemisystems we have constructed are indeed new.

\section{Preliminaries}\label{section:prelim}

\subsection{Generalized quadrangles and $m$-ovoids}
A generalized quadrangle of order $(s,t)$ is a point-line incidence structure obeying the following axioms.
\begin{itemize}
 \item Any two points are incident with at most one line.
\item Every line is incident with $s+1$ points.
\item Every point is incident with $t+1$ lines.
\item Given a point $P$ and a line $\ell$ that are not incident, there is a unique point $Q$ on
  $\ell$ that is collinear to $P$.
\end{itemize}

The dual of a generalized quadrangle of order $(s,t)$ is a generalized quadrangle of order $(t,s)$.
The family of generalized quadrangles that we are mainly interested in are the Hermitian surfaces
$\herm(3,q^2)$, where $q$ is a prime power. We define $\herm(3,q^2)$ to be comprised of the set of
totally isotropic points and lines of a non-degenerate Hermitian form on $\PG(3,q^2)$; which results
in a generalized quadrangle of order $(q^2,q)$. The dual of $\herm(3,q^2)$ is a generalized
quadrangle of order $(q,q^2)$, isomorphic to the geometry of totally singular points and lines of an
elliptic quadric $\cQ^-(5,q)$ arising from a non-singular quadratic form of minus type on
$\PG(5,q)$.

We will be working almost exclusively in this dual setting. Let $m\geq 1$ be an integer. A set of
points is said to be an \emph{$m$-ovoid} of $\cQ^-(5,q)$ if every line of $\cQ^-(5,q)$ meets the set
in $m$ points. Note that an $m$-ovoid is the dual concept to an $m$-cover of lines (i.e., upon
interchanging the roles of points and lines), and in particular from the above results, a nontrivial
$m$-ovoid of $\cQ^-(5,q)$ must have $m= (q+1)/2$. The following lemma follows directly from
\cite[Lemma 1]{Bamberg:2007aa}.

\begin{lemma}\label{res_1} 
Let $\mathcal{M}$ be a set of $m(q^3+1)$ points in $\cQ^{-}(5,q)$. Then, $\mathcal{M}$ is an
$m$-ovoid of $\cQ^{-}(5,q)$ if and only if
\begin{align*}
|P^\perp\cap\mathcal{M}|=\begin{cases}(m-1)(q^2+1)+1,\quad &\textup{if $P\in\mathcal{M}$},\\
m(q^2+1),\quad & \textup{otherwise}.
\end{cases}
\end{align*}
Here $\perp$ is the polarity defined by the elliptic quadric $\cQ^-(5,q)$.
\end{lemma}

 The above result makes it possible to use projective two-intersection sets for constructing
 $m$-ovoids. A two-intersection set $\mathcal{K}$ is a set of points in $\PG(n,q)$ such that every
 hyperplane of $\PG(n,q)$ is incident with either $h_1$ or $h_2$ points of $\mathcal{K}$. We call
 $h_1$ and $h_2$ \emph{intersection numbers}. A related concept to a two-intersection set is an
 intriguing set. A set $\mathcal{I}$ of points of a generalized quadrangle is called
 \emph{intriguing} if there are integers $k_1$ and $k_2$ such that the number of points of
 $\mathcal{I}$ collinear to an arbitrary point $P$ of the generalized quadrangle is $k_1$ if $P\in
 \mathcal{I}$, and $k_2$ otherwise.

\subsection{Strongly regular graphs and Cayley graphs}

A $(v,k,\lambda,\mu)$ {\it strongly regular graph} is a simple undirected regular graph on $v$
vertices with valency $k$ satisfying the following: for any two adjacent (resp. nonadjacent)
vertices $x$ and $y$ there are exactly $\lambda$ (resp. $\mu$) vertices adjacent to both $x$ and
$y$. It is known that a graph with valency $k$, not complete or edgeless, is strongly regular if and
only if its adjacency matrix has exactly two restricted eigenvalues. Here, we say that an eigenvalue
of the adjacency matrix is {\it restricted} if it has an eigenvector perpendicular to the all-ones
vector.

Let $G$ be a finite abelian group and $D$ be an inverse-closed subset of $G\setminus\{0\}$. We
define a graph $\Cay(G,D)$ with the elements of $G$ as its vertices; two vertices $x$ and $y$ are
adjacent if and only if $x-y\in D$. The graph $\Cay(G,D)$ is called a {\it Cayley graph} on $G$ with
connection set $D$. The eigenvalues of $\Cay(G,D)$ are given by $\psi(D)$, $\psi\in \dualG$, where
$\dualG$ is the \emph{dual group} consisting of all characters of $G$. Using the aforementioned
spectral characterization of strongly regular graphs, we see that $\Cay(G,D)$ with connection set
$D$($\not=\varnothing,G$) is strongly regular if and only if $\psi(D)$, $\psi\in
\dualG\setminus\{1\}$, take exactly two values, say $\alpha_1$ and $\alpha_2$. We note that if
$\Cay(G,D)$ is strongly regular with two restricted eigenvalues $\alpha_1$ and $\alpha_2$, then the
sets $\{\psi\in \dualG\colon \psi(D)=\alpha_i\}$, $i=1,2$, also form connection sets of strongly
regular Cayley graphs on $\dualG$; one is the complement of another in $\dualG\setminus\{1\}$, and
each of these sets is called the {\it dual} of $D$.

For a nonzero vector $x\in \F_q^6$, we use $\langle x\rangle$ to denote the projective point in
$\PG(5,q)$ corresponding to the one-dimensional subspace over $\F_q$ spanned by $x$. In this paper,
we will use the following relation between certain intriguing sets and strongly regular graphs: For
an intriguing set $\cM$ in $\cQ^{-}(5,q)$, define $D:=\{\lambda x\colon \lambda\in\F_q^\ast,\; \la x
\ra\in \cM\}$, which is a subset of $(\F_{q}^6,+)$. Then the Cayley graph with vertex set
$(\F_q^6,+)$ and connection set $D$ is strongly regular. Its restricted eigenvalues can be
determined as follows. Let $\psi$ be a nontrivial additive character of $\F_q^6$. Then $\psi$ is
principal on a unique hyperplane $P^\perp$ for some $P\in \PG(5,q)$. We have
\begin{align*}
\psi(D)&=\sum_{\la x\ra\in\cM}\sum_{\lambda\in\F_q^*}\psi(\lambda x)
=\sum_{\la x\ra\in\cM}(q\KD (\la x\ra)-1)\\
&=-|\cM|+q|P^\perp\cap\cM|
=\begin{cases}-q^3+m(q-1),\; &\textup{if $P\in\mathcal{M}$},\\
m(q-1),\; & \textup{otherwise,}\end{cases}
\end{align*}
where for a subset $S$ of the points, $\mathds{1}_{S}$ is the characteristic function taking value 1
on elements of $S$ and value $0$ elsewhere. Conversely, for each hyperplane $P^\perp$ of $\PG(5,q)$,
we can find a nontrivial character $\psi$ that is principal on $P^\perp$, and the size of
$P^\perp\cap \cM$ can be computed from $\psi(D)$. Therefore, the character values of $D$ reflect the
intersection properties of $\cM$ with the hyperplanes of $\PG(5,q)$. To summarize, we have the
following result.

\begin{result}\label{res_charD} 
Let $\mathcal{M}$ be a set of $m(q^3+1)$ points in $\cQ^{-}(5,q)$. 
Define
\begin{equation}\label{eq:defD}
D:=\{\lambda x\colon \lambda\in\F_q^\ast,\;\la x \ra\in \cM\}\subset (\F_q^6,+).
\end{equation}
Then, $\mathcal{M}$  is an  $m$-ovoid of  $\cQ^{-}(5,q)$ if and only if 
for any $P\in \PG(5,q)$
\begin{align*}
\psi(D)
=\begin{cases}-q^3+m(q-1),\; &\textup{if $P\in\mathcal{M}$},\\
m(q-1),\; & \textup{otherwise,}\end{cases}
\end{align*}
where $\psi$ is any nontrivial character of $\F_q^6$ that is principal on the hyperplane $P^\perp$.
\end{result}

\subsection{A finite field model of the elliptic quadric $\cQ^-(5,q)$}

We will use the following model of $\cQ^{-}(5,q)$. We view $\F_{q^6}$ as a 6-dimensional vector
space over $\F_q$. We define the trace function $\tr_{q^n/q}:\mathbb{F}_{q^n} \rightarrow
\mathbb{F}_q$ by $\tr_{q^n/q}(x) = x + x^q + x^{q^2} +\cdots + x^{q^{n-1}}$. Define a quadratic form
$Q: \F_{q^6}\rightarrow \F_{q}$ by
\[
Q(x):=\tr_{q^3/q}(x^{q^3+1}).
\]
The quadratic form $Q$ is clearly elliptic and the projective points corresponding to the nonzero
vectors of $\{x\in \F_{q^6}\mid Q(x)=0\}$ form an elliptic quadric. This will be our model for
$\cQ^-(5,q)$. Note that for a point $P=\la x\ra$, its polar hyperplane $P^\perp$ is given by
$P^\perp=\{\la y \ra\colon \tr_{q^6/q}(yx^{q^3})=0\}$.

Let $\psi_{\F_{q^6}}$ and $\psi_{\F_q}$ be the canonical additive characters of $\F_{q^6}$ and
$\F_q$, respectively. Then, each nontrivial additive character $\psi_a$ of $\F_{q^6}$ has the form
\begin{equation}\label{fieldchara}
\psi_{a}(x)=\psi_{\F_{q^6}}(ax)=\psi_{\F_{q}}(\Tr_{q^6/q}(ax)),\; \; x\in \F_{q^6},
\end{equation}
where $a\in \F_{q^6}^\ast$. Since $\psi_{a}$ is principal on the hyperplane $\{\la x \ra:
\Tr_{q^6/q}(ax)=0\}=P^{\perp}$ with $P=\la a^{q^3}\ra$, the character sum condition in
Result~\ref{res_charD} can be more explicitly rewritten as
\begin{equation}\label{eqn:polarD}
\psi_{a}(D)
=\begin{cases}-q^3+m(q-1),\quad &\textup{if $a^{q^3}\in D$},\\
m(q-1),\quad & \textup{otherwise.}\end{cases}
\end{equation}

\subsection{Gauss sums} We need some preparation for computing (additive) character values of a
subset of vectors of $\F_{q^m}$. For a multiplicative character $\chi$ and the canonical additive
character $\psi$ of $\F_q$, define the {\it Gauss sum} by
\[
G_q(\chi)=\sum_{x\in \F_q^\ast}\chi(x)\psi(x).
\]
The following are some basic properties of Gauss sums:
\begin{enumerate}
\item[(i)] $G_q(\chi)\overline{G_q(\chi)}=q$ if $\chi$ is nontrivial;
\item[(ii)] $G_q(\chi^{-1})=\chi(-1)\overline{G_q(\chi)}$;
\item[(iii)] $G_q(\chi)=-1$ if $\chi$ is trivial.
\end{enumerate}
Let $\gamma$ be a fixed primitive element of $\F_q$ and $k$ a positive integer dividing $q-1$. For
$0\le i\le k-1$ we set $C_i^{(k,q)}=\gamma^i C_0$, where $C_0$ is the subgroup of index $k$ of
$\F_q^\ast$. The {\it Gauss periods} associated with these cyclotomic classes are defined by
$\psi(C_i^{(k,q)}):=\sum_{x\in C_i^{(k,q)}}\psi(x)$, $0\le i\le k-1$, where $\psi$ is the canonical
additive character of $\F_q$. As described in the introduction, since we take a union of cyclotomic
classes of index $k=4(q^2+q+1)$ of $\F_{q^6}$ as a subset $D$ of \eqref{eq:defD}, we need to compute
a sum of Gauss periods. By orthogonality of characters, the Gauss periods can be expressed as a
linear combination of Gauss sums:
\begin{equation}
\psi(C_i^{(k,q)})=\frac{1}{k}\sum_{j=0}^{k-1}G_q(\chi^{j})\chi^{-j}(\gamma^i), \; 0\le i\le k-1,
\end{equation}
where $\chi$ is any fixed multiplicative character of order $k$ of $\F_q$.  

\begin{theorem}\label{thm:Yama}{\em (\cite[Theorem~1]{Yama})}
Let $\chi$ be a nontrivial multiplicative character of $\F_{q^m}$ and $\chi'$ be its restriction to
$\F_q$. Take a system $L$ of representatives of $\F_{q^m}^\ast/\F_q^\ast$ such that $\Tr_{q^m/q}$
maps $L$ onto $\{0,1\}\subset\F_q$. Partition $L$ into two parts:
\[
L_0=\{x\in L\colon \Tr_{q^m/q}(x)=0\} \, \, \mbox{and}\, \, L_1=\{x\in L\colon \Tr_{q^m/q}(x)=1\}.
\] 
Then, 
\[
\sum_{x\in L_1}\chi(x)=\left\{
\begin{array}{ll}
G_{q^m}(\chi)/G_q(\chi'), & \mbox{ if   $\chi'$ is nontrivial,}\\
-G_{q^m}(\chi)/q,& \mbox{ otherwise}. 
 \end{array}
\right.
\]
\end{theorem}

\begin{theorem}\label{thm:semiprim}{\em (\cite[Theorem~11.6.3]{BEW})}
Let $p$ be a prime. Suppose that $m>2$ and $p$ is semi-primitive modulo $m$, i.e., there exists a
positive integer $s$ such that $p^s\equiv -1\pmod{m}$. Choose $s$ minimal and write $f=2st$ for any
positive integer $t$. Let $\chi_m$ be a multiplicative character of order $m$ of $\F_{p^f}$. Then,
\[
p^{-f/2}G_{p^f}(\chi_m)=
\left\{
\begin{array}{ll}
(-1)^{t-1},&  \mbox{if $p=2$,}\\
(-1)^{t-1+(p^s+1)t/m},&  \mbox{if $p>2$. }
 \end{array}
\right.
\]
\end{theorem}
%%%%%%%%%%%%%%%%%%%%%%%%%%%%%%%%%%%%%%%%%%%%%%%%%%%
We will need the  {\it Davenport-Hasse lifting formula}, which is stated below.  
\begin{theorem}\label{thm:lift}
{\em (\cite[Theorem~11.5.2]{BEW})} Let $\chi'$ be a nontrivial multiplicative character of
$\F_{p^f}$ and let $\chi$ be the lift of $\chi'$ to $\F_{p^{fs}}$, i.e.,
$\chi(\alpha)=\chi'(\Norm_{p^{fs}/p^f}(\alpha))$ for $\alpha\in \F_{p^{fs}}$, where $s\geq 2$ is an
integer. Then
\[
G_{p^{fs}}(\chi)=(-1)^{s-1}(G_{p^f}(\chi'))^s. 
\]
\end{theorem}
The following theorem is often referred to as the  {\it Davenport-Hasse product formula}.  
\begin{theorem}
\label{thm:Stickel2}{\em (\cite[Theorem~11.3.5]{BEW})} Let $\eta$ be a multiplicative character of
order $\ell>1$ of $\F_{p^f}$. For every nontrivial multiplicative character $\chi$ of $\F_{p^f}$,
\[
G_{p^f}(\chi)=\frac{G_{p^f}(\chi^\ell)}{\chi^\ell(\ell)}
\prod_{i=1}^{\ell-1}
\frac{G_{p^f}(\eta^i)}{G_{p^f}(\chi\eta^i)}. 
\]
\end{theorem}

%%%%%%%%%%%%%%%%%%%%%%%%%%%%%%%%%%%%%%%%%%%%%%%%%%%%%%
The following is the main theorem of this section. 
\begin{theorem}\label{Gaussmain}
Let $q=p^f$ be an odd prime power such that $q\equiv 3\pmod{4}$, and let $m$ be an odd positive
integer dividing $N=q^2+q+1$. Let $\chi_m'$ be a multiplicative character of order $m$ of $\F_{q^3}$
and $\chi_m$ be its lift to $\F_{q^6}$, and $\chi_4$ be a multiplicative character of order $4$ of
$\F_{q^6}$. Then, it holds that $G_{q^6}(\chi_4\chi_{m})=G_{q^6}(\chi_4^3\chi_{m})$. In particular,
it holds that
\begin{equation}\label{eq:Gaussmain}
G_{q^6}(\chi_4\chi_{m})=\rho_q G_{q^3}({\chi'}_{m}^{4})G_{q^3}({\chi'}_{m}^{-2}), 
\end{equation}
where $\rho_q=-1$ or $1$ depending on whether $q\equiv 3\pmod{8}$ or 
 $q\equiv 7\pmod{8}$. 
\end{theorem} 
\proof
First, we have
\[
G_{q^6}(\chi_4\chi_{m})=G_{q^6}(\chi_4^{q^3}\chi_{m}^{q^3})=
G_{q^6}(\chi_4^3\chi_{m}). 
\]

Applying the Davenport-Hasse product formula (Theorem~\ref{thm:Stickel2}) with 
$\ell=4$, 
$\chi=\chi_{4}\chi_m$, and $\eta=\chi_4$, we have 
\begin{align}
G_{q^6}(\chi_{4}\chi_m)&=
\frac{G_{q^6}(\chi_m^{4})G_{q^6}(\chi_4)G_{q^6}(\chi_4^2)G_{q^6}(\chi_4^3)}{\chi_m^{4}(4)G_{q^6}(\chi_4^2\chi_m)G_{q^6}(\chi_4^3\chi_m)G_{q^6}(\chi_m)}\nonumber\\
&=
q^6\frac{G_{q^6}(\chi_m^{4})G_{q^6}(\chi_4^2)}{G_{q^6}(\chi_4^2\chi_m)G_{q^6}(\chi_4\chi_m)G_{q^6}(\chi_m)}. \label{inse}
\end{align}
On the other hand, we have 
\begin{equation}\label{outs}
G_{q^6}(\chi_4^2\chi_m)=
\frac{G_{q^6}(\chi_m^{2})G_{q^6}(\chi_4^2)}{\chi_m^{2}(2)G_{q^6}(\chi_m)}=\frac{G_{q^6}(\chi_m^{2})G_{q^6}(\chi_4^2)}{
G_{q^6}(\chi_m)}. 
\end{equation}
By substituting \eqref{outs} into \eqref{inse}, we have
\begin{equation}\label{eq:Gauss2}
G_{q^6}(\chi_{4}\chi_m)^2=q^6\frac{G_{q^6}(\chi_m^{4})}{G_{q^6}(\chi_m^{2})}=
G_{q^6}(\chi_m^{4})G_{q^6}(\chi_m^{-2}). 
\end{equation}
Then, by the Davenport-Hasse lifting formula (Theorem~\ref{thm:lift}), Eq.~(\ref{eq:Gauss2}) is reformulated as 
\[
G_{q^6}(\chi_{4}\chi_m)^2=G_{q^3}({\chi'}_m^{4})^2G_{q^3}({\chi'}_m^{-2})^2,
\]
i.e., 
\[
G_{q^6}(\chi_{4}\chi_m)=\rho G_{q^3}({\chi'}_m^{4})G_{q^3}({\chi'}_m^{-2})
\]
for some $\rho\in \{-1,1\}$. 

Now, we determine the sign of $\rho$ by induction. Write $m=\ell p_1$, where $\ell$ is a positive
integer and $p_1$ is an odd prime. First we consider the case where $\ell=1$, i.e., $m=p_1$. Take
the reduction of $G_{q^6}(\chi_{4}\chi_{p_1})^{p_1}$ modulo $p_1$:
\begin{align}
G_{q^6}(\chi_{4}\chi_{p_1})^{p_1}&\equiv \sum_{z\in \F_{q^{6}}^\ast}\chi_{4}^{p_1}\chi_{p_1}^{p_1 }(z)\psi(p_1z)\pmod{p_1}\nonumber\\
&=\sum_{z\in \F_{q^{6}}^\ast}\chi_{4}^{-p_1}(p_1)\chi_{4}^{p_1}(p_1z)\psi(p_1z)\nonumber\\
&=G_{q^6}(\chi_{4}^{p_1})=G_{q^6}(\chi_{4}). \label{eq:Gauss4}
\end{align}
By Theorem~\ref{thm:semiprim}, we have $G_{q^6}(\chi_4)=\rho_q q^3$, where $\rho_q=-1$ or $1$
depending on whether $q\equiv 3\pmod{8}$ or $q\equiv 7\pmod{8}$. On the other hand,
\begin{align}
G_{q^6}(\chi_{4}\chi_{p_1})^{p_1}&=\left(\rho  G_{q^3}({\chi'}_{p_1}^{4})G_{q^3}({\chi'}_{p_1}^{-2})\right)^{p_1}\nonumber\\
&= \rho^{p_1} G_{q^3}({\chi'}_{p_1}^{4})^{p_1} G_{q^3}({\chi'}_{p_1}^{-2})^{p_1}
\equiv  
\rho\pmod{p_1}.\label{eq:Gauss5}
\end{align}
Hence, by  Eqs.~(\ref{eq:Gauss4}) and (\ref{eq:Gauss5}), we have 
$\rho\equiv \rho_q q^3\pmod{p_1}$. Since 
$p_1\,|\,(q^3-1)$, we obtain $\rho=\rho_q$.

We next consider the case where $\ell>1$. Write $\ell =p_2 \ell'$ with 
$p_2$ a prime, and 
let  $\chi_{\ell' p_1}:=\chi_{m}^{p_2}$.  Assume that 
\[
G_{q^6}(\chi_{4}\chi_{\ell' p_1})=\rho_q G_{q^3}({\chi'}_{\ell' p_1}^{4})G_{q^3}({\chi'}_{\ell' p_1}^{-2}). 
\]
Then, we have
\[
G_{q^6}(\chi_{4}\chi_m)^{p_2}\equiv 
\sum_{z\in \F_{q^{6}}^\ast}\chi_{4}^{p_2}\chi_m^{p_2 }(z)\psi(p_2 z)\,  \pmod{p_2}=G_{q^6}(\chi_{4}\chi_{\ell' p_1}). 
\]
On the other hand, 
\[
G_{q^6}(\chi_{4}\chi_m)^{p_2}=  
\rho G_{q^3}({\chi'}_m^{4})^{p_2}G_{q^3}({\chi'}_m^{-2})^{p_2}
\equiv  \rho G_{q^3}({\chi'}_{\ell' p_1}^{4})G_{q^3}({\chi'}_{\ell' p_1}^{-2})\pmod{p_2}.
\]
This implies that $\rho=\rho_q$. 
This completes the proof of the theorem. 
\qed

\begin{corollary}\label{cor:Gauss}
With the same notation as in Theorem~\ref{Gaussmain}, it holds that
\[
G_{q^6}(\chi_4\chi_m)=\rho_q \frac{q^3 G_{q^3}(\chi_2' {\chi'}_m^{2}) }{G_{q^3}(\chi_2')}, 
\]
where $\chi_2'$ is the quadratic character of $\F_{q^3}$. 
\end{corollary}
%%%%%%%%%%%%%%%%%%%%%%%%%%%%%%%%%%%%%%%%%%%%%%%%%%%%%%%%%%
\proof 
Applying the Davenport-Hasse product formula with $\ell=2$, $\chi=\chi'_2{\chi'}_{m}^{2}$ and $\eta=\chi_2'$, we have 
\[
G_{q^3}(\chi_2'{\chi'}_m^{2})=\frac{G_{q^3}(\chi_2')}{q^3}G_{q^3}({\chi'}_m^{4})G_{q^3}({\chi'}_m^{-2}). 
\]
Then, Eq.~\eqref{eq:Gaussmain} of Theorem~\ref{Gaussmain} is reformulated as
\[
G_{q^6}(\chi_4\chi_m)=\rho_q \frac{q^3 G_{q^3}(\chi_2' {\chi'}_m^{2}) }{G_{q^3}(\chi_2')}.
\] 
This completes the proof. 
\qed
%%%%%%%%%%%%%%%%%%%%%%%%%%%%%%%%%%%%%%%%%%%%%%%%%%%%%%%%
%%%%%%%%%%%%%%%%%%%%%%%%%%%%%%%%%%%%%%%%%%%%%%%%%%%%%%%%
%%%%%%%%%%%%%%%%%%%%%%%%%%%%%%%%%%%%%%%%%%%%%%%%%%%%%%%%

\section{The beginnings of a construction: a partition of a conic in $\PG(2,q)$}\label{section:conic}

Let $\omega$ be a primitive element of $\F_{q^3}$ and $N:=q^2+q+1$. Viewing $\F_{q^3}$ as a
3-dimensional vector space over $\F_q$, we will use $\F_{q^3}$ as the underlying vector space of
$\PG(2,q)$. The points of $\PG(2,q)$ are $\langle \omega^i\rangle$, $0\le i\le N-1$, and the lines
of $\PG(2,q)$ are
\begin{equation}\label{eqn_Lu}
L_c:=\{\langle x\rangle\colon \Tr_{q^3/q}(\omega^c x)=0\},
\end{equation}
where $0\le c\le N-1$. Of course, $\la \omega^i\ra=\la \omega^{i+jN}\ra$ and $L_{c}=L_{c+jN}$, for
any $i,j$ and $c$.

Define a quadratic form $f: \F_{q^3}\rightarrow \F_q$ by $f(x):= \tr_{q^3/q}(x^2)$. The associated
bilinear form $B: \F_{q^3}\times \F_{q^3}\rightarrow \F_{q}$ is given by $B(x,y)=2\tr_{q^3/q}(xy)$.
It is clear that $B$ is nondegenerate. Therefore $f$ defines a conic $\cQ$ in $\PG(2,q)$, which
contains $q+1$ points. Consequently each line $\l$ of $\PG(2,q)$ meets $\cQ$ in $0$, $1$ or $2$
points, and $\l$ is called an {\it exterior}, {\it tangent} or {\it secant line} accordingly. Also
it is known that each point $P\in \PG(2,q) \setminus \cQ$ is on either $0$ or $2$ tangent lines to
$\cQ$, and $P$ is called an {\it interior} or {\it exterior point} accordingly.

Consider the following subset of $\Z_N$: 
\begin{equation}\label{eqn_IQ}
I_\cQ:=\{i\colon 0\le i\le N-1,\,\tr_{q^3/q}(\omega^{2i})=0\}=\{d_0,d_1,\ldots, d_{q}\},
\end{equation}
where the elements are numbered in any (unspecified) order. That is, $\cQ=\{\langle
\omega^{d_i}\rangle\colon 0\le i\le q\}$. Furthermore, consider the following subset (a so-called
{\it Singer difference set}) of $\Z_N$:
\begin{equation}\label{eqn_Sin}
S:=\{ i \pmod*{N} \colon \Tr_{q^3/q}(\omega^i)=0\}. 	
\end{equation}
That is, $L_0=\{\langle \omega^{i}\rangle\colon i \in S\}$. Then, it is clear that $I_\cQ\equiv
2^{-1}S\pmod{N}$.

For $x\in \F_{q^3}$, we define the sign of $x$, $\sgn(x)\in \{0,1,-1\}$, by 
\begin{equation}\label{eqn:sgn}
\sgn(x)
=\begin{cases} 1,\quad &\textup{if $x$ is a nonzero square},\\
-1,\quad &\textup{if $x$ is a nonsquare},\\
0,\quad & \textup{if $x=0$.}\end{cases}
\end{equation}
\begin{lemma}\label{lem:conicpro}
With the above notation, we have the following. 
\begin{enumerate}
\item[(1)] The polarity of $\PG(2,q)$ induced by $\cQ$ interchanges $\langle \omega^c\rangle $ and
  $L_c$. In particular, it maps points on $\cQ$ to tangent lines, and exterior (resp. interior)
  points to secant (resp. exterior) lines.
\item[(2)] For any point $P=\langle x\rangle$ off $\cQ$, $P$ is  exterior 
(resp. interior) if and only if $\sgn(f(x))=\epsilon$ (resp. $-\epsilon$), where
$\epsilon=1$ or $-1$ depending on whether $q\equiv 1\pmod{4}$ or $3\pmod{4}$. 
\end{enumerate}
\end{lemma}
\proof For the proof of (1), we refer the reader to \cite{h1}. For the proof of (2) in the case
where $q\equiv 1\pmod{4}$, see \cite[Lemma 3.3]{Feng:2015aa}. The case $q\equiv 3\pmod{4}$ can be
proved in a similar way. \qed

\vspace{0.3cm}
By Lemma~\ref{lem:conicpro}, 
\begin{align*}
\mbox{ $L_c$ is tangent} & \mbox{ $\Leftrightarrow$ $\sgn(f(\omega^c))=0$\hspace{3mm} $\Leftrightarrow$ $|(S-c)\cap I_{\cQ}|=1$, }\\
\mbox{ $L_c$ is exterior} & \mbox{ $\Leftrightarrow$ $\sgn(f(\omega^c))=-\epsilon$ $\Leftrightarrow$ $|(S-c)\cap I_{\cQ}|=0$, }\\
\mbox{ $L_c$ is secant\, } & \mbox{ $\Leftrightarrow$ $\sgn(f(\omega^c))=\epsilon$\hspace{3.3mm} $\Leftrightarrow$ $|(S-c)\cap I_{\cQ}|=2$. }
\end{align*}
Define $D_1:=\bigcup_{i\in I_\cQ}C_{i}^{(N,q^3)}$, where 
$C_{i}^{(N,q^3)}$ is represented by $\langle \omega^i\rangle$. 
Then, $D_1$ takes exactly three nontrivial character values:  
\begin{align}
\sum_{i \in I_\cQ}\psi_{\F_{q^3}}(\omega^{c} C_i^{(N,q^3)})
=&\,\sum_{i \in I_\cQ}\psi_{\F_{q}}(\Tr_{q^3/q}(\omega^{c+i}) \F_{q}^\ast)
=\,q|(S-c)\cap I_\cQ|-(q+1)\nonumber\\
=&\,
\left\{
\begin{array}{ll}
-1,&  \mbox{if $c\pmod*{N}\in I_\cQ$,}\\
-1+\epsilon q,&  \mbox{if 
$c\pmod*{N}  \in I_s$, }\\
-1-\epsilon q,&  \mbox{if 
$c\pmod*{N} \in I_n$,}
 \end{array}
\right.\label{eqn:sgn2}
\end{align}
where $I_s:=\{i\pmod{N}\colon \Tr_{q^3/q}(\omega^{2i})\in C_0^{(2,q)}\}$ and $I_n:=\{i\pmod{N}\colon
\Tr_{q^3/q}(\omega^{2i})\in C_1^{(2,q)}\}$.

\begin{remark}\label{rem:conic}
We define the following subsets of $\F_{q^3}$: 
\[
D_0:=\{0\}, \, 
D_1:=\bigcup_{i\in I_\cQ}C_{i}^{(N,q^3)}, \, D_2:=\bigcup_{i\in I_s}C_{i}^{(N,q^3)}, \,  D_3:=\bigcup_{i\in I_n}C_{i}^{(N,q^3)}. 
\]
In the language of association schemes, the Cayley graphs $\Cay(\F_{q^3},D_i)$, $i=0,1,2,3$, form a
three-class association scheme on $\F_{q^3}$. See \cite{BST}.
\end{remark}

%%%%%%%%%%%%%%%%%%%%%%%%%%%%%%%%%%%%%%%%%%%%%%
We will consider a partition of  $D_1$. 
For $d_0\in I_\cQ$, we define
\[
{\mathcal X}:=\{\omega^{d_i}\Tr_{q^3/q}(\omega^{d_0+d_i})\colon 1\le i\le q\}\cup\{2  \omega^{d_0}\}
\]
and 
\[
X:=\{\log_{\omega}(x)\pmod*{2N}\colon x\in {\mathcal X}\}\subset \Z_{2N}. 
\]
It is clear that $X\equiv I_\cQ\pmod{N}$. 
We list a few properties of the 
set $X$ below. 
\begin{remark}\leavevmode
\begin{enumerate} 
\item[(i)] {\rm (\cite[Lemma~3.4]{Feng:2015aa})} If we use any other $d_i$ in place of $d_0$ in
  the definition of ${\mathcal X}$, then the resulting set $X'$ satisfies $X'\equiv X\pmod{2N}$ or
  $X'\equiv X+N\pmod{2N}$.
\item[(ii)] {\rm (\cite[Remark 3.5]{Feng:2015aa})} The set $X$ is invariant under multiplication by $q$ modulo $2N$.
\end{enumerate}
\end{remark}
The set $X$ was used to construct $\frac{(q^2-1)}{2}$-tight sets of $\cQ^+(5,q)$ in \cite{Feng:2015aa}.
We note that $\frac{(q^2-1)}{2}$-tight sets of $\cQ^+(5,q)$ were independently constructed in \cite{De-Beule:2016yu}. Surprisingly, $X$ is also behind our new $(q+1)/2$-ovoids of $\cQ^-(5,q)$.

The set $X$ can be expressed as 
\begin{equation}\label{eqn_defX}
X=2S_1''\cup (2S_2''+N)\pmod{2N}
\end{equation}
for some $S_1'',S_2''\subseteq \Z_N$ with $|S_1''|+|S_2''|=q+1$. That is, we are partitioning $X$
into the {\it even} and {\it odd} parts. Define $S_i'\equiv 2S_i''\pmod{N}$ and $S_i\equiv
2S_i'\pmod{N}$ for $i=1,2$. Then, $S_1'\cup S_2'\equiv I_\cQ\pmod{N}$ and $S_1\cup S_2\equiv
S\pmod{N}$, i.e., $X$ induces partitions of the conic $\cQ$ and the line $L_0$ respectively. We will
use this partition $S_1,S_2$ of $S$ to define our $(q+1)/2$-ovoids in the next section. Consider the
following partition of $D_1$:
\[
D_{1,1}:=\bigcup_{i\in X}C_{i}^{(2N,q^3)} \mbox{\,  and \, }  D_{1,2}:=\bigcup_{i\in X+N}C_{i}^{(2N,q^3)}.  
\] 
\begin{theorem} {\rm(\cite[Theorem 3.7, Remark 3.8]{Feng:2015aa})} \label{thm:main2}
With notation as above, the set $D_{1,1}$ takes exactly four nontrivial character values, that is,
\[
\psi_{\F_{q^3}}(\omega^c D_{1,1})=\left\{
\begin{array}{ll}
\frac{-1+\eta(2)qG_q(\eta)}{2}, & \mbox{ if $c \pmod*{N} \in I_\cQ$ and 
$c \pmod*{2N} \in X$,}\\
\frac{-1-\eta(2)qG_q(\eta)}{2}, & \mbox{ if $c \pmod*{N} \in I_\cQ$ and $c \pmod*{2N} \in X+N$,}\\
\frac{-1+\epsilon q}{2}, & \mbox{ if $c \pmod*{N} \in I_s$}, \\
\frac{-1-\epsilon q}{2}, & \mbox{ if $c \pmod*{N} \in I_n$,}
 \end{array}
\right.
\]
where $\eta$ is the quadratic character of $\F_q$. 
\end{theorem}
\begin{remark}\label{rem:conic2}\
\begin{enumerate}
\item[(i)] In \cite{Feng:2015aa}, the authors treated only the case where $q\equiv 1\pmod{4}$ of
  Theorem~\ref{thm:main2}. The case $q\equiv 3\pmod{4}$ can be proved in a similar way.
\item[(ii)] In the language of association schemes, the Cayley graphs $\Cay(\F_{q^3},D_i)$,
  $i=0,2,3$, and $\Cay(\F_{q^3},D_{1,j})$, $j=1,2$, form a four-class association scheme on
  $\F_{q^3}$, which is a fission scheme of that of Remark~\ref{rem:conic}.
\end{enumerate}
\end{remark}

\section{New $\frac{q+1}{2}$-ovoids of $\cQ^{-}(5,q)$}\label{section:construction} In the rest of
this paper, we assume that $q\equiv 3\pmod{4}$ is a prime power. In this section, we give a
construction of $\frac{q+1}{2}$-ovoids of $\cQ^{-}(5,q)$.

\subsection{Construction of $\frac{q+1}{2}$-ovoids of $\cQ^{-}(5,q)$}
Consider the following bilinear form from $\F_{q^6}^2$ to $\F_q$: 
 \[
B(x,y):=\Tr_{q^6/q}(xy^{q^3}). 
\]
This form is symmetric and defines an elliptic orthogonal space isomorphic to $\cQ^{-}(5,q)$, where
the associated quadratic form is given by $Q(x)=\Tr_{q^3/q}(x^{q^3+1})$. We now define a subset $D$
of the elliptic quadric $\{x\in\F_{q^6}^\ast\colon \Tr_{q^3/q}(x^{q^3+1})=0\}$.

\begin{construction}\label{mainconstruction}
Let $S_1, S_2$ be the partition of the Singer difference set $S$ defined by $S_1'',S_2''$ of
\eqref{eqn_defX}. Let $J_1:=\{0,3\}$ and $J_2:=\{1,2\}$, and put
\[
I:=\{Ni-(q+1)j \hspace{-0.2cm}\pmod{4N}: (i,j)\in (J_1\times S_1) \cup (J_2\times S_2)\}.  
\]
Now, define 
\[
D:=\bigcup_{i \in I}C_i^{(4N,q^6)},
\]
where $C_i^{(4N,q^6)}:=\gamma^i C_0$ with $C_0$ the subgroup of index $4N$ of $\F_{q^6}^\ast$ and
$\gamma$ a fixed primitive element of $\F_{q^6}$ such that $\gamma^{q^3+1}=\omega$ (where $\omega$
was defined in Section \ref{section:conic}).
\end{construction}

It is clear that $|D|=(q^3+1)(q^2-1)/2$. Furthermore, $D$ is a subset of $\{x\in \F_{q^6}\colon
\Tr_{q^3/q}(x^{q^3+1})=0\}$. In fact, for any $x=\gamma^{Ni-(q+1)j+s}\in D$ with $\gamma^s\in
C_0^{(4N,q^6)}$
\begin{align*}
\Tr_{q^3/q}(\gamma^{(Ni-(q+1)j+s)(q^3+1)})&\,=
\Tr_{q^3/q}(\omega^{Ni-(q+1)j+s})\\
&\, =\omega^{Ni+s-Nj}\Tr_{q^3/q}(\omega^{Nj-(q+1)j})\\
&\, =\omega^{Ni+s-Nj}\Tr_{q^3/q}(\omega^{j})=0. 
\end{align*}

The following is our main theorem which will be proved in the next section. 
\begin{theorem}\label{thm:main}
The set $\cM$ of points in $\PG(5,q)$ corresponding to $D$ (defined in Construction
\ref{mainconstruction}) forms a $(q+1)/2$-ovoid of $\cQ^{-}(5,q)$.
\end{theorem}

\subsection{Computations of character values} By Result~\ref{res_charD}, we only need to show that
the (additive) character values of $D$ take the prescribed values. In particular, it suffices to
show that for all $\gamma^a\in \F_{q^6}^\ast$,
\[
\psi_{\F_{q^6}}(\gamma^a D)
=\begin{cases}-q^3+\frac{q^2-1}{2},\quad &\textup{if $\gamma^{a q^3}\in D$},\\
\frac{q^2-1}{2},\quad & \textup{otherwise, }\end{cases}
\]
where $\psi_{\F_{q^6}}$ is the canonical additive character of $\F_{q^6}$. 

Let $\chi_{4}$, $\chi_{N}$, and $\chi_{4N}$ be multiplicative characters of order $4$, $N$, and $4N$
of $\F_{q^6}$, respectively. By the orthogonality of characters, we have
\begin{equation}\label{eq:fi1}
\psi_{\F_{q^6}}(\gamma^a D)=\frac{1}{4N}\sum_{h=0}^{4N-1}G_{q^6}(\chi_{4N}^h)\sum_{i\in I}\chi_{4N}^{-h}(\gamma^{a+i}). 
\end{equation}
Since $\gcd{(4,N)}=1$, $\chi_{4N}^h$ is uniquely expressed as $\chi_{4}^{h_1}\chi_{N}^{h_2}$ for some  $(h_1,h_2)\in 
\Z_{4}\times \Z_N$. Then,   the right hand side of Eq.~\eqref{eq:fi1} is rewritten as 
\begin{align}
&\,\frac{1}{4N}
\sum_{h_1=0,1,2,3}\sum_{h_2=0}^{N-1}
G_{q^6}(\chi_{4}^{h_1}\chi_{N}^{h_2})
\big(\sum_{j\in J_1}\sum_{s\in S_1}\chi_{4}^{-h_1}(\gamma^{a+Nj})\chi_{N}^{-h_2}(\gamma^{a-(q+1)s})\nonumber\\
&\hspace{3cm} 
+\sum_{j\in J_2}\sum_{s\in S_2}\chi_{4}^{-h_1}(\gamma^{a+Nj})\chi_{N}^{-h_2}(\gamma^{a-(q+1)s})\big).\label{eq:init}
\end{align}
By noting that each $S_i$ is invariant under multiplication by $q$ modulo $N$, we have  
\begin{align}
\psi_{\F_{q^6}}(\gamma^a D)=&\,
\frac{1}{4N}
\sum_{h_1=0,1,2,3}\sum_{h_2=0}^{N-1}
G_{q^6}(\chi_{4}^{h_1}\chi_{N}^{h_2})
\big(\sum_{j\in \{0,3\}}\sum_{s\in S_1}\chi_{4}^{-h_1}(\gamma^{a+Nj})\chi_{N}^{-h_2}(\gamma^{a+q^2 s})\nonumber\\
&\hspace{1cm}
+\sum_{j\in \{1,2\}}\sum_{s\in S_2}\chi_{4}^{-h_1}(\gamma^{a+Nj})\chi_{N}^{-h_2}(\gamma^{a+q^2 s})\big) \nonumber\\
=&\,
\frac{1}{4N}
\sum_{h_1=0,1,2,3}\sum_{h_2=0}^{N-1}
G_{q^6}(\chi_{4}^{h_1}\chi_{N}^{h_2})
\big(\sum_{j\in \{0,3\}}\sum_{s\in S_1}\chi_{4}^{-h_1}(\gamma^{a+Nj})\chi_{N}^{-h_2}(\gamma^{a+ s})\nonumber\\
&\hspace{1cm} 
+\sum_{j\in \{1,2\}}\sum_{s\in S_2}\chi_{4}^{-h_1}(\gamma^{a+Nj})\chi_{N}^{-h_2}(\gamma^{a+ s})\big).
\label{eq:ppppart} 
\end{align}

We compute the right hand side of Eq.~\eqref{eq:ppppart} by dividing it into the three partial sums:
$P_1,P_2$ and $P_3$, where $P_1$ is the contribution of the summands with $h_1=0$, $P_2$ is the
contribution of the summands with $h_1=2$, and $P_3$ is the contribution of the summands with
$h_1=1$ or $3$. That is, we have
\begin{equation}\label{eq:p1p2p3}
\psi_{\F_{q^6}}(\gamma^a D)=P_1+P_2+P_3. 
\end{equation}
It is clear that $P_2=0$ since \[
\sum_{j\in \{0,3\}}\chi_4^{-2}(\gamma^{a+Nj})=\sum_{j\in \{1,2\}}\chi_4^{-2}(\gamma^{a+Nj})=0. 
\]
We consider the partial sum $P_1$. 
\begin{lemma}\label{lem:res1}
It holds that 
\[
P_1=\left\{
\begin{array}{ll}
\frac{-q^3+q^2-1}{2},&  \mbox{if $a\in S \pmod*{N}$,}\\
\frac{q^2-1}{2},&  \mbox{if $a\not\in S \pmod*{N}$. }
 \end{array}
\right.
\]
\end{lemma}
\proof 
We now compute 
\begin{equation}
 P_1= 
\frac{1}{2N}\sum_{h_2=0}^{N-1}
G_{q^6}(\chi_{N}^{h_2})
\sum_{s\in S}\chi_{N}^{-h_2}(\gamma^{a+ s}). \label{eq:quadpart}
\end{equation}
Let $\chi_N'$ be the multiplicative character of order $N $ of $\F_{q^3}$ such that 
$\chi_N$ is the lift of $\chi_N'$. 
Since  
\[
G_{q^3}({\chi'}_N^{-h_2})=q \sum_{s\in S}{\chi'}_{N}^{-h_2}(\omega^{s})=q \sum_{s\in S}\chi_{N}^{-h_2}(\gamma^{s})
\]
and 
\[
G_{q^6}(\chi_{N}^{h_2})=-G_{q^3}({\chi'}_N^{h_2})^2
\]
by 
Theorems~\ref{thm:Yama} and \ref{thm:lift}, respectively,  
continuing from  \eqref{eq:quadpart}, we have 
\begin{align*}
P_1+\frac{q+1}{2N}=&\, -\frac{1}{2Nq}\sum_{h_2=1}^{N-1}
G_{q^3}({\chi'}_N^{h_2})^2
G_{q^3}({\chi'}_N^{-h_2})
{\chi'}_{N}^{-h_2}(\omega^{a})\\
=&\,  -\frac{q^2}{2N}\sum_{h_2=1}^{N-1}
G_{q^3}({\chi'}_N^{h_2})
{\chi'}_{N}^{-h_2}(\omega^{a})=-\frac{q^3}{2N}\sum_{h_2=0}^{N-1}
\sum_{s\in S}{\chi'}_{N}^{-h_2}(\omega^{-s+a})+\frac{q^3(q+1)}{2N}\\
=&\, \left\{
\begin{array}{ll}
\frac{q^3(q+1)}{2N}-\frac{q^3}{2},&  \mbox{if $a\in S \pmod*{N}$,}\\
\frac{q^3(q+1)}{2N},&  \mbox{if $a\not\in S \pmod*{N}$. }
 \end{array}
\right.
\end{align*}
The conclusion of the lemma now follows. \qed 

\vspace{0.3cm} Next, we evaluate the partial sum $P_3$. Recall that $S_i'\equiv 2^{-1}S_i\pmod{N}$,
$S_i''\equiv 2^{-1}S_i'\pmod{N}$, $i=1,2$, and
\[
X=2S_1''\cup (2S_2''+N) \pmod{2N}. 
\]
\begin{lemma}\label{lem:res2}
Let 
$b\equiv 4^{-1}a\pmod{N}$ and $c\equiv 2b\pmod{2N}$. 
Then, it holds that 
\begin{equation}\label{eq:p2}
P_3=\frac{\rho_q \delta_a q^3}{G_{q^3}({\chi'}_2)}
\psi_{\F_{q^3}}(\omega^c\bigcup_{t\in X}C_t^{(2N,q^3)})-\frac{\rho_q \delta_a q^3}{2G_{q^3}({\chi'}_2)}\psi_{\F_{q^3}}(\omega^c \bigcup_{t\in I_\cQ}C_t^{(N,q^3)}). 
\end{equation}
where $\psi_{\F_{q^3}}$ is the canonical additive character of $\F_{q^3}$ and $\delta_a=1$ or $-1$
depending on whether $a\equiv 0,1\pmod{4}$ or $a\equiv 2,3\pmod{4}$.
\end{lemma}
\proof 
We have 
\begin{align}
P_3=&\,\frac{1}{4N}
\sum_{h_2=0}^{N-1}
G_{q^6}(\chi_{4}\chi_{N}^{h_2})
\big(\sum_{j\in \{0,3\}}\sum_{s\in S_1}\chi_{4}^{-1}(\gamma^{a+Nj})\chi_{N}^{-h_2}(\gamma^{a+s})\nonumber\\ 
&\hspace{2cm}+\sum_{j\in \{1,2\}}\sum_{s\in S_2}\chi_{4}^{-1}(\gamma^{a+Nj})\chi_{N}^{-h_2}(\gamma^{a+s})\big)\nonumber\\
&\hspace{0.6cm}+\frac{1}{4N}\sum_{h_2=0}^{N-1}
G_{q^6}(\chi_{4}^3\chi_{N}^{h_2})
\big(\sum_{j\in \{0,3\}}\sum_{s\in S_1}\chi_{4}(\gamma^{a+Nj})\chi_{N}^{-h_2}(\gamma^{a+s})\nonumber\\ 
&\hspace{2cm}+\sum_{j\in \{1,2\}}\sum_{s\in S_2}\chi_{4}(\gamma^{a+Nj})\chi_{N}^{-h_2}(\gamma^{a+s})\big)\nonumber
\end{align}
Noting that $G_{q^6}(\chi_{4}\chi_{N}^{h_2})=G_{q^6}(\chi_{4}^3\chi_{N}^{h_2})$ by Theorem~2.7, we have
\begin{align}
P_3=&\,\frac{\delta_a}{2N}
\sum_{h_2=1}^{N-1}
G_{q^6}(\chi_{4}\chi_{N}^{h_2})
\big(\sum_{s\in S_1}\chi_{N}^{-h_2}(\gamma^{a+s})-\sum_{s\in S_2}\chi_{N}^{-h_2}(\gamma^{a+s})\big)\nonumber\\
&\hspace{0.6cm} +\frac{\delta_a}{2N}
G_{q^6}(\chi_{4})(|S_1|-|S_2|), \label{eq:fi2}
\end{align}
where $\delta_a=1$ or $-1$ depending on whether $a\equiv 0,1\pmod{4}$ or $a\equiv 2,3\pmod{4}$. Note
that $G_{q^6}(\chi_4)=\rho_q q^3$ by Theorem~\ref{thm:semiprim}. We now compute the former summand
of \eqref{eq:fi2}. Applying Theorem~\ref{Gaussmain}, we have
\begin{align}
&
\sum_{h_2=1}^{N-1}
G_{q^6}(\chi_{4}\chi_{N}^{h_2})
\big(\sum_{s\in S_1}\chi_{N}^{-h_2}(\gamma^{a+s})-\sum_{s\in S_2}\chi_{N}^{-h_2}(\gamma^{a+s})\big)\nonumber\\
=&\, \frac{\rho_q q^3}{G_{q^3}(\chi_2')}\sum_{h_2=1}^{N-1}
G_{q^3}(\chi_2'{\chi'}_N^{2h_2})
\big(\sum_{s\in S_1}{\chi'}_{N}^{-h_2}(\omega^{a+s})-\sum_{s\in S_2}{\chi'}_{N}^{-h_2}(\omega^{a+s})\big)\nonumber\\
=&\,\frac{\rho_q q^3}{G_{q^3}(\chi_2')}\sum_{h_2=1}^{N-1}
G_{q^3}(\chi_2'{\chi'}_N^{2h_2}){\chi'}_{N}^{-2h_2}(\omega^{2b})
\big(\sum_{s\in S_1'}{\chi'}_{N}^{-2h_2}(\omega^{s})-\sum_{s\in S_2'}{\chi'}_{N}^{-2h_2}(\omega^{s})\big)\nonumber\\
=&\,\frac{\rho_q q^3}{G_{q^3}({\chi'}_2)}\sum_{h_2=1}^{N-1}
G_{q^3}(\chi_2'{\chi'}_N^{2h_2})
\sum_{t\in X}\chi_2'{\chi'}_{N}^{-2h_2}(\omega^{t+c}). \label{eq:cent}
\end{align}
Put the value of \eqref{eq:cent} as $T$. 
Let $\chi_{2N}':=\chi_2'\chi_N'$, which is a multiplicative character of order $2N$ of $\F_{q^3}$. 
Then, we have 
\begin{align}
T=&\,\frac{\rho_q q^3}{G_{q^3}({\chi'}_2)}\sum_{h: \tiny{\mbox{odd}};\,h\not=N}
G_{q^3}({\chi'}_{2N}^{h})
\sum_{t\in X}{\chi'}_{2N}^{-h}(\omega^{t+c})\nonumber
\\
=&\,\frac{\rho_q  q^3}{G_{q^3}({\chi'}_2)}\sum_{h=0}^{2N-1}
G_{q^3}({\chi'}_{2N}^{h})
\sum_{t\in X}{\chi'}_{2N}^{-h}(\omega^{t+c})\nonumber\\
&\hspace{0.6cm} -\frac{\rho_q  q^3}{G_{q^3}({\chi'}_2)}
\sum_{h=0}^{N-1}G_{q^3}({\chi'}_{2N}^{2h})
\sum_{t\in X}{\chi'}_{2N}^{-2h}(\omega^{t+c})-
\frac{\rho_q  q^3}{G_{q^3}({\chi'}_2)}
G_{q^3}({\chi'}_{2N}^{N})
\sum_{t\in X}{\chi'}_{2N}^{-N}(\omega^{t+c}). \label{eq:lastpart}
\end{align}
By the orthogonality of characters, we have 
\[
\sum_{h=0}^{2N-1}
G_{q^3}({\chi'}_{2N}^{h})
\sum_{t\in X}{\chi'}_{2N}^{-h}(\omega^{t+c})=2N\psi_{\F_{q^3}}(\omega^c\bigcup_{t\in X}C_t^{(2N,q^3)}). 
\]
Furthermore, noting that $X\equiv I_\cQ\pmod{N}$, we have 
\begin{align*}
\sum_{h=0}^{N-1}G_{q^3}({\chi'}_{2N}^{2h})
\sum_{t\in X}{\chi'}_{2N}^{-2h}(\omega^{t+c})
= N\psi_{\F_{q^3}}(\omega^c \bigcup_{t\in I_\cQ}C_t^{(N,q^3)}). 
\end{align*}
Finally, the last term of \eqref{eq:lastpart} is computed as 
\[
-\rho_q q^3(|S_1|-|S_2|). 
\]
Summing up, we have \eqref{eq:p2} of this lemma. 
 \qed

\vspace{0.3cm}
We  are now ready to prove our main theorem.

\begin{proof}[Proof of Theorem~\ref{thm:main}]
Recall that $\psi_{\F_{q^6}}(\gamma^a D)=P_1+P_2+P_3$ as in \eqref{eq:p1p2p3}, where $P_2=0$, and
$P_1$ and $P_3$ are computed in Lemmas~\ref{lem:res1} and \ref{lem:res2}, respectively. By
\eqref{eqn:sgn2} and Theorem~\ref{thm:main2}, we have
\begin{align*}
\psi_{\F_{q^6}}(\gamma^a D)=&\, \left\{
\begin{array}{ll}
\frac{-q^3+q^2-1}{2},&  \mbox{if $a\pmod*{N}\in S $,}\\
\frac{q^2-1}{2},&  \mbox{if $a\pmod*{N}\not\in S $,}
 \end{array}
\right.\\
&\hspace{0.3cm}+\frac{\rho_q \delta_a q^3}{G_{q^3}({\chi'}_2)}
\psi_{\F_{q^3}}(\omega^c\bigcup_{t\in X}C_t^{(2N,q^3)})-\frac{\rho_q \delta_a q^3}{2G_{q^3}({\chi'}_2)}\psi_{\F_{q^3}}(\omega^c \bigcup_{t\in I_\cQ}C_t^{(N,q^3)})\\
=&\, \left\{
\begin{array}{ll}
\frac{-q^3+q^2-1}{2},&  \mbox{if $a\pmod*{N}\in S $,}\\
\frac{q^2-1}{2},&  \mbox{if $a\pmod*{N}\not\in S $,}
 \end{array}
\right.\\
&\hspace{0.3cm}+\frac{\rho_q \delta_a q^3}{G_{q^3}({\chi'}_2)}\cdot 
\left\{
\begin{array}{ll}
\frac{-1+\eta(2)qG_q(\eta)}{2},& \mbox{ if  $c\pmod*{N}\in I_\cQ$ and $c\pmod*{2N}\in X$,}\\
\frac{-1-\eta(2)qG_q(\eta)}{2},& \mbox{ if  $c\pmod*{N}\in I_\cQ$ and $c\pmod*{2N} \in X+N$,}\\
-\frac{q+1}{2},& \mbox{ if $c\pmod*{N}\in I_s$}, \\
\frac{q-1}{2},& \mbox{ if $c\pmod*{N}\in I_n$},
 \end{array}
\right. \\
&\hspace{0.3cm}-\frac{\rho_q \delta_a q^3}{2G_{q^3}({\chi'}_2)}\cdot 
\left\{
\begin{array}{ll}
-1,&  \mbox{if $c\pmod*{N}\in I_\cQ$,}\\
-q-1,&  \mbox{if $c\pmod*{N}\in I_s$,}\\
q-1,&  \mbox{if $c\pmod*{N} \in I_n$.}
 \end{array}
\right.
\end{align*}
We remark the following facts:   
\begin{align*}
c\pmod*{N}\in I_\cQ&\, \, \Leftrightarrow \, \, a\pmod*{N}\in S,\\
c\pmod*{2N}\in X&\, \, \Leftrightarrow \, \, a\pmod*{N}\in S_1,\\
c\pmod*{2N}\in X+N&\, \, \Leftrightarrow \, \, a\pmod*{N}\in S_2.
\end{align*}
Then, we have 
\begin{align*}
\psi_{\F_{q^6}}(\gamma^a D)=&\, \left\{
\begin{array}{ll}
\frac{-q^3+q^2-1}{2}+\frac{\rho_q \delta_a \eta(2)q^4G_q(\eta)}{2G_{q^3}(\chi_2')},& \mbox{ if $a\pmod*{N}\in S_1$}, \\
\frac{-q^3+q^2-1}{2}-\frac{\rho_q \delta_a \eta(2)q^4G_q(\eta)}{2G_{q^3}(\chi_2')},& \mbox{ if $a\pmod*{N}\in S_2$,}\\
\frac{q^2-1}{2},&  \mbox{  if $a\pmod*{N}\not\in S$.}
 \end{array}
\right.
\end{align*} 
Note that $\eta(2)=-1$ or $1$ depending on $q\equiv 3\pmod{8}$ or 
$q\equiv 7\pmod{8}$. Furthermore, by $G_{q^3}(\chi_2')^2=-q^3$ and 
$G_{q^3}(\chi_2')=G_{q}(\eta)^3$, we have
\[
\frac{\rho_q \eta(2)q^4G_q(\eta)}{G_{q^3}(\chi_2')}=-qG_q(\eta)G_{q^3}(\chi_2')=
-qG_q(\eta)^4=-q^3. 
\]
Hence, 
\[
\psi_{\F_{q^6}}(\gamma^a D)=\left\{
\begin{array}{ll}
-q^3+\frac{q^2-1}{2},& \mbox{ if $a\pmod*{N}\in S_1$ and $a\equiv 0,1\pmod*{4}$,}\\
& \mbox{\hspace{0.4cm} or $a\pmod*{N}\in S_2$ and $a\equiv 2,3\pmod*{4}$}, \\
\frac{q^2-1}{2},&  \mbox{  if $a\pmod*{N}\not\in S$, $a\pmod*{N}\in S_1$ and $a\equiv 2,3\pmod*{4}$,}\\
&\mbox{\hspace{0.4cm}  or $a\pmod*{N}\in S_2$ and  $a\equiv 0,1\pmod*{4}$.}
 \end{array}
\right.
\]
Thus, $D$ takes exactly two nontrivial character values, i.e., the Cayley graph $\Cay(\F_{q^6},D)$ is strongly regular. 

Finally, we show that $\psi_{\F_{q^6}}(\gamma^a D)=-q^3+\frac{q^2-1}{2}$ if and only if
$\gamma^{aq^3}\in D$. To do this, we determine the dual of $D$ explicitly. Let $K_1:=\{0,1\}$ and
$K_2:=\{2,3\}$. Define
\[
J:=\{Ni-(q+1)j\pmod*{4N}\colon  (i,j)\in (K_1\times S_1) \cup (K_2\times S_2)\}  
\]
and $E:=\bigcup_{i \in J}C_i^{(4N,q^6)}$. Then, $E$ is obviously the dual of $D$. Hence,
$\psi_{\F_{q^6}}(\gamma^a D)=-q^3+\frac{q^2-1}{2}$ if and only if $\gamma^{a}\in E$. Since $q^3
I\equiv J\pmod{4N}$, we obtain the assertion. The proof of the theorem is now complete.
\end{proof}

\section{On groups and equivalence with known examples}\label{section:equivalence}

From our construction, it is not difficult to identify a subgroup of the stabilizer of 
a hemisystem arising from Construction \ref{mainconstruction}.

\begin{theorem}
Let $q$ be a prime power congruent to $3$ modulo $4$. Then the hemisystem $\mathcal{M}$ of
$\cQ^-(5,q)$ arising from Construction \ref{mainconstruction} is stabilized by a subgroup of
$\mathsf{P\Gamma O}^-(6,q)$ isomorphic to the metacyclic group $C_{(q^3+1)/4}:C_3$. Moreover:
\begin{enumerate}[(i)]
\item The normal cyclic subgroup of order $C_{(q^3+1)/4}$ is induced by right multiplication by
  $\gamma^{4(q^2+q+1)}$, where $\gamma$ is a primitive element of $\mathbb{F}_{q^6}$.
\item The complementary element of order $3$ on top arises from the map
$x\mapsto x^{q^2}$.
\end{enumerate}
\end{theorem}

The known infinite families of examples are
\begin{itemize}
\item The Cossidente-Penttila examples, with each example stabilized by $\mathsf{P\Sigma L}(2, q^2)$;
\item The BGR-hemisystems arising from the Bamberg-Giudici-Royle (BGR) construction from \cite{Bamberg:2010aa}.
\end{itemize}

The order of $\mathsf{P\Sigma L}(2,q^2)$ is $q^2(q^4-1) \log_p(q)$. So in particular, it is not
divisible by $q^2-q+1$. The generic BGR construction yields many hemisystems, including the
Cossidente-Penttila ones. Those that are not of Cossidente-Penttila type have a normal subgroup of
order $q^2$.

\begin{theorem}[{\cite[Theorem 3.3]{Bamberg:2013aa}}]
Let $\mathcal{H}$ be a BGR-hemisystem of $\herm(3,q^2)$. Then the full stabilizer of $\mathcal{H}$
contains $T\rtimes K$ where $T$ is an elementary abelian group of order $q^2$ and $K$ is a subgroup
of $\mathsf{Sp}(4,q)$.
\end{theorem}

For many reasons, the exceptional case is $q=3$. Here the stabilizer of the Segre hemisystem is
$\mathsf{PSL}(3,4).2$ in its exceptional embedding in $\mathsf{P\Gamma U}(4,3)$. The order of
$\mathsf{PSL}(3,4)$ is $20160=3^2\cdot (3^2-3+1)\cdot 320$, and it will turn out to be the only
occasion when a subgroup $M$ of $\mathsf{PGU}(4,q)$ has order divisible by $q^2(q^2-q+1)$, apart
from subgroups containing $\mathsf{SU}(3,q)$. It is not difficult to see that no hemisystems admit
$\mathsf{SU}(3,q)$, since this group acts transitively on totally isotropic lines of $\herm(3,q^2)$.

The following lemma follows directly from \cite[Theorem 4.2]{Bamberg:2008ab}.

\begin{lemma}
Let $q$ be an odd prime power with $q\ge 5$. Let $M$ be a maximal subgroup of $\mathsf{PGU}(4,q)$
with order divisible by $q^2-q+1$. Then $M$ is the stabilizer of a non-degenerate hyperplane and is
isomorphic to $\mathsf{GU}(3,q)$.
\end{lemma}

Now the order of $\mathsf{GU}(3,q)$ is $q^3(q+1)^2(q^2-1)(q^2-q+1)$. By \cite[Theorem
4.1]{Bamberg:2008ab}, we have the following:
\begin{lemma}
Let $q$ be an odd prime power with $q\ge 5$. Let $M$ be a maximal subgroup of $\mathsf{GU}(3,q)$
with order divisible by $q^2-q+1$. Then one of the following occurs:
\begin{enumerate}[(i)]
\item $\mathsf{SU}(3,q)\trianglelefteq M$;
\item $q=5$ and $M\cong 6.S_7$; 
\item $M\cong \mathsf{\Gamma U}(1,q^3)$.
\end{enumerate}
\end{lemma}

So if $M$ is a maximal subgroup of $\mathsf{PGU}(4,q)$, $q$ odd and $q\ge 5$, with order divisible
by $q^2(q^2-q+1)$, then $M$ contains $\mathsf{SU}(3,q)$. Therefore, the $m$-ovoids we have
constructed in this paper are not of BGR type for $q\ge 5$.

\section{Open problems}

In the introduction, we mentioned that there are examples of hemisystems of $\herm(3,q^2)$ admitting
a cyclic group of order $q^2-q+1$ for most of the small values of $q$. In \cite[Section
4.1]{Bamberg:2013aa}, the computational data, including the orders of the stabilisers of the
examples, was presented and which we repeat below:

\begin{center}

\begin{tabular}{c|c|c}
$q$&$q^2-q+1$&Stabiliser\\
\hline
3&7&$\mathrm{PSL}(3,4). 2$\\
5&21&$3\cdot A_7\cdot 2$\\
7&43&$43: 6$\\
9&73&$73:6$\\
11&111&$111:6$, $333:3$\\
17&273&$273: 3$\\
19&343&$1715: 6$\\
23&507&$507: 6$\\
27&703& $703: 3$\\
\hline
\end{tabular}
\end{center}

Notice that there were no examples for $q=13$, nor for $q=25$, and so one might believe that there
are no examples for $q\equiv 1\pmod{12}$. In this paper, we have given a construction for each $q$
congruent to $3\pmod 4$, and so an open problem remains whether a similar construction can work for
$q\equiv 5,9\pmod{12}$, say. We have attempted to adapt our methods for $q$ congruent to
$1\pmod{4}$, but to no avail. The essential problem is the computation of Gauss sums. In the case
where $q\equiv 3\pmod{4}$, our hemisystem admits $C_{(q^3+1)/4}$, and it was enough to consider
Gauss sums of order $4(q^2+q+1)$. But, known examples of hemisystems for $q\equiv 1\pmod 4$ seem to admit
the smaller cyclic group $C_{q^2-q+1}$ instead. Thus, relative to the $3$ modulo $4$ case, we need
to compute Gauss sums of larger order. We could not find an effective way to compute the character
values, and furthermore, we do not know what kind of structure is behind the examples. The authors
 believe a more enlightening distinction could be made in this case: the examples could be further
divided according to $q$ modulo $3$. It seems the $q\equiv 0\pmod{3}$ examples take on a different
nature, and in particular, the field automorphisms fix a unique orbit of $C_{q^2-q+1}$. Hence, we
give the following refinement of the open problem given in \cite[Section 4.1]{Bamberg:2013aa}:

\begin{problem}
Does there exist a hemisystem invariant under a cyclic group of order $q^2-q+1$ for each odd prime
power $q$ satisfying $q\equiv 1\pmod{4}$ and $q\equiv 0,2\pmod{3}$?
\end{problem}

%\bibliographystyle{abbrv}
%\bibliography{refs}

\begin{thebibliography}{99}

\bibitem{Bamberg:2010aa}
J.~Bamberg, M.~Giudici, and G.~F. Royle.
\newblock Every flock generalized quadrangle has a hemisystem.
\newblock {\em Bull. Lond. Math. Soc.}, 42(5):795--810, 2010.

\bibitem{Bamberg:2013aa}
J.~Bamberg, M.~Giudici, and G.~F. Royle.
\newblock Hemisystems of small flock generalized quadrangles.
\newblock {\em Des. Codes Cryptogr.}, 67(1):137--157, 2013.

\bibitem{Bamberg:2007aa}
J.~Bamberg, S.~Kelly, M.~Law, and T.~Penttila.
\newblock Tight sets and {$m$}-ovoids of finite polar spaces.
\newblock {\em J. Combin. Theory Ser. A}, 114(7):1293--1314, 2007.

\bibitem{Bamberg:2008ab}
J.~Bamberg and T.~Penttila.
\newblock Overgroups of cyclic {S}ylow subgroups of linear groups.
\newblock {\em Comm. Algebra}, 36(7):2503--2543, 2008.

\bibitem{BST}
E.~Bannai, O.~Shimabukuro, and H.~Tanaka.
\newblock Finite euclidean graphs and ramanujan graphs.
\newblock {\em Discrete Mathematics}, 309(20):6126--6134, 2009.

\bibitem{BEW}
B.~C. Berndt, R.~J. Evans, and K.~S. Williams.
\newblock {\em Gauss and {J}acobi sums}.
\newblock Canadian Mathematical Society Series of Monographs and Advanced
  Texts. John Wiley \& Sons, Inc., New York, 1998.
\newblock A Wiley-Interscience Publication.

\bibitem{Bruen:1978qf}
A.~A. Bruen and J.~W.~P. Hirschfeld.
\newblock Applications of line geometry over finite fields. {II}. {T}he
  {H}ermitian surface.
\newblock {\em Geom. Dedicata}, 7(3):333--353, 1978.

\bibitem{Cossidente:2005aa}
A.~Cossidente and T.~Penttila.
\newblock Hemisystems on the {H}ermitian surface.
\newblock {\em J. London Math. Soc. (2)}, 72(3):731--741, 2005.

\bibitem{De-Beule:2016yu}
J.~De~Beule, J.~Demeyer, K.~Metsch, and M.~Rodgers.
\newblock A new family of tight sets in {$\mathcal{Q}^+(5,q)$}.
\newblock {\em Des. Codes Cryptogr.}, 78(3):655--678, 2016.

\bibitem{Feng:2015aa}
T.~Feng, K.~Momihara, and Q.~Xiang.
\newblock Cameron-{L}iebler line classes with parameter {$x=\frac{q^2-1}{2}$}.
\newblock {\em J. Combin. Theory Ser. A}, 133:307--338, 2015.

\bibitem{h1}
J.~W.~P. Hirschfeld.
\newblock {\em Projective geometries over finite fields}.
\newblock The Clarendon Press, Oxford University Press, New York, 1979.
\newblock Oxford Mathematical Monographs.

\bibitem{Segre:1965aa}
B.~Segre.
\newblock Forme e geometrie hermitiane, con particolare riguardo al caso
  finito.
\newblock {\em Ann. Mat. Pura Appl. (4)}, 70:1--201, 1965.

\bibitem{Thas:1995aa}
J.~A. Thas.
\newblock Projective geometry over a finite field.
\newblock In {\em Handbook of incidence geometry}, pages 295--347.
  North-Holland, Amsterdam, 1995.

\bibitem{Dam:2013aa}
E.~R. van Dam, W.~J. Martin, and M.~Muzychuk.
\newblock Uniformity in association schemes and coherent configurations:
  cometric {Q}-antipodal schemes and linked systems.
\newblock {\em J. Combin. Theory Ser. A}, 120(7):1401--1439, 2013.

\bibitem{Yama}
K.~Yamamoto and M.~Yamada.
\newblock Williamson {H}adamard matrices and {G}auss sums.
\newblock {\em J. Math. Soc. Japan}, 37(4):703--717, 1985.

\end{thebibliography}

\end{document}